\documentclass[preprint,12pt]{article}
 \usepackage{amsthm,amsmath,amssymb}
\usepackage{natbib}
\pdfoutput=1

\usepackage{tikz}
\usepackage{comment}
\newtheorem{thm}{Theorem}
\newtheorem{prop}{Proposition}
\newtheorem{lem}{Lemma}
\newtheorem{rem}{Remark}
\usepackage{amssymb}
\usepackage{float}
\usepackage{bm}
\usepackage{graphicx}
\usepackage{dsfont}
\newcommand{\ind}[1]{\,\mathds{1}_{#1}}
\usepackage{enumerate}

\newcommand{\E}{\mathbb{E}}
\usepackage{pst-all}	

\usepackage[margin=0.75in]{geometry}
\usepackage{pgfplots}
\pgfplotsset{width=10cm,compat=1.9}
\usetikzlibrary{patterns}
\usepgfplotslibrary{external}

\title{Asymptotics for irregularly observed long memory processes}

\author{M. Ould Haye$^1$\footnote{corresponding author : {mohamedouhaye@cunet.carleton.ca}}, A. Philippe$^2$
}
\date{$^1$\small School of Mathematics and Statistics. \\
Carleton University, 1125 Colonel By Dr. Ottawa, ON, Canada, K1S 5B6
\\ $^2$ Laboratoire de Math\'{e}matiques Jean Leray,\\
2 rue de la Houssiniere, Universit\'{e} de Nantes, 44 322 Nantes France.}

\begin{document}

\maketitle
\begin{abstract}
We study the effect of observing a long memory stationary process at irregular time points via a renewal process. We establish a sharp difference in the asymptotic behaviour of the self-normalized sample mean of the observed process depending on the renewal process. In particular, we show that if the renewal process has a moderate heavy tail distribution, then the limit is a so-called Normal Variance Mixture (NVM) and we characterize the randomized variance part of the limiting NVM as an integral function of a L\'evy stable motion. Otherwise, the normalized sample mean will be asymptotically normal.
\end{abstract}

\section{Introduction}\label{Introduction}

Modeling and analyzing irregularly observed time series data is a classic problem encountered across various fields, including astronomy (\cite{10.1093rastirzac011}), signal processing (\cite{SUN2024110075}),
finance (\cite {huang2024generative}), environmental sciences (\cite{Beelaerts2010TimeSeriesRF}), and biomedical sciences (\cite{Shukla2018ModelingIS}). 
Additionally, generating irregularly observed time series is utilized in differential privacy methods to enhance data privacy (e.g., \cite{pmlr-v151-koga22a}).

The statistical treatment of irregularly observed data
 can be traced back to \cite{Lomb1976LeastsquaresFA}, \cite{Parzen1984TimeSA}, \cite{10.1007/978-1-4684-9403-7_8} and \cite{JONES1985157} where Kalman filter and maximum likelihood-based methods were proposed to handle the irregularity in observing the data and were efficient in the case of missing data.
 
 A common approach consists in embedding the 
observed process $X_{t_k},k=1,\ldots,n$, where $t_k$ are the observed time points, into a continuous-time process $X_t, t>0$. However, there are few results available for inference from randomly observed long memory time series data with the exception of Gaussian processes (\cite{bardet}). Most of the existing literature focuses on inference from deterministic sampling of continuous-time processes (see \cite{tsai1,tsai2}, \cite{comte1}, \cite{comte2}, 
\cite{chambers}).
A different context appears when the random time
ignoring the irregular sampling of the process $X$ and considering the discrete-time process $Y_k$ 
\begin{equation}\label{sampled1}
Y_k=X_{T_k},\qquad k=1,2,\ldots
\end{equation}
to investigate what information on $X$ can be retrieved without observing $T_k$ of a renewal process $T$.
Recent publications have pointed out some of the challenges, one can be facing when working with such data, inherent to the very nature of the lack of regularity of their occurrences as such regularity has been the cornerstone of time series analysis. For instance, \cite{PhilippeRobetViano} showed that while stationarity is preserved, if one samples from a Gaussian process then the observed process will no longer be Gaussian. Despite the loss of normality of the process, the normalized partial sum process converges to a fractional Brownian motion, provided that the renewal process has a finite first moment.
\cite{philippe-viano2008} showed that using a renewal process with infinite moment can drastically reduce the memory and result in a sampled process with short-range dependence if the original process has a long-range one.
\cite{OuldHaye2023} proved that the local Whittle estimator remains a consistent estimator of the long-memory parameter of $X$ when the renewal process is a Poisson process and $X$ is Gaussian. 

In this paper, we investigate large sample properties of $Y_k$ defined in (\ref{sampled1}) when $X$ is of long memory without necessarily being Gaussian. 
We establish a sharp difference in the asymptotic behaviour of the self-normalized sample mean of the observed process depending on the renewal process that is being used to model the unevenly observed data. As we will see, it all amounts to a competition between the tail distribution parameter $\alpha$ and the so-called memory parameter $d$. As shown in Theorem 1 below, if the renewal process has a finite moment (e.g., a Poisson process) or, conversely, has a 'very heavy' tail distribution, the limiting distribution is the same—namely, a normal distribution. However, if the renewal process has an 'intermediate heavy' tail distribution, the limiting distribution becomes a Normal Variance Mixture (NVM), which is the product $\sqrt{Z}N$ of two independent random variables where $Z$ is positive (often referred to as randomized variance) and $N$ has the standard normal distribution. More details on NVM can be found in \cite{HINTZ2021107175} and the references therein. We provide a complete characterization of the randomized variance component of the limiting NVM as an integral function of a L\'evy stable subordinator. Additionally, we establish the continuity of the limiting distribution as it transitions between different levels of tail heaviness (i.e., from very heavy to intermediate heavy tails). Several auxiliary results, which may be of independent interest, are also presented. Specifically, we derive exact asymptotic expressions for the covariance function of the observed process and the asymptotic expectation of harmonic means for a class of distributions supported on the unit interval, extending earlier results for the uniform distribution. 

The remainder of the paper is organized as follows: In Section 2, we present a technical lemma concerning the renewal process $T_n$ and derive exact asymptotic expressions for the covariance function of the sampled process $Y_n$ under general weak stationarity assumptions when $T_n$ has infinite moments. Section 3 contains the main theorem and its proof. The Appendix provides the proofs of the technical results.
\section{Sampling from a covariance-stationary process}
We consider a covariance-stationary process $X_t$ where the time index can be either discrete or continuous.
(i.e., $t=1,2,\ldots,$, or $t\ge0$) with a covariance function $\sigma_X$ that behaves asymptotically as
\begin{equation}\label{cov_stat}
\sigma_X(h)\sim \tilde C_dh^{2d-1},\qquad\textrm{as }h\to\infty,\qquad0<d<\frac{1}{2},
\end{equation}
for some positive constant $\tilde C_d$. This implies
\begin{equation}\label{exact1}
\sigma_X(h)=u(h+1)(h+1)^{2d-1},
\end{equation}
where $u$ is bounded, say by $K$, nonzero, $u(h)\to\tilde C_d$ as $h\to\infty$.
These processes are referred to as long-memory processes as their covariance function decays very slowly with increasing lag $h$ making it non-summable.
The parameter $d$ is called the memory parameter of the process $X_t$. For further details on long-memory processes, see for instance \cite{surgailis2012large}.

Now, assume the renewal process $T$ is given by
$$T_k=\Delta_1+\cdots+\Delta_k,$$ 
where $\Delta_j$ are i.i.d. positive random variables. 
 The purpose of this Section is to derive the asymptotic expressions for the covariance function of the sampled process $Y_n=X_{T_n}$ (as defined in (\ref{sampled1})) when the sampling renewal process has heavy tail distribution with infinite first moment. In \cite{philippe-viano2008} and \cite{PhilippeRobetViano}, exact expressions were derived for the case when $\mathbb{E}(\Delta_1)<\infty$, while only asymptotic bounds were found when $\mathbb{E}(\Delta_1)=\infty.$\\
 Assume that the tail probability of $\Delta_1$ is given by
\begin{equation}\label{heavy tail} 
P(\Delta_1\ge x)=x^{-\alpha}\ell(x),\qquad\textrm{for all }x\ge1,
\end{equation}
where $0<\alpha\le1$ and $\ell$ a slowly varying function at infinity. Note that in the case of continuous time index,  $\Delta_1$ can have the entire $(0,\infty)$ as support but one is interested on the tail probability function on $[1,\infty)$ only.
Note also that equation above results in the fact that for all $1\le x\le a<\infty$,
\begin{equation}\label{karamata1}
0<P(\Delta_1\ge a)\le\ell(x)\le a^\alpha<\infty
\end{equation}
and therefore $\ell(x)$ is locally bounded away from zero and infinity on $[1,\infty)$.
For any positive integer $n$, define the function
\begin{equation}\label{elstar}
\ell^*(n)=\int_1^n\frac{\ell(x)}{x}dx.
\end{equation}
By a special version of Karamata Theorem (see \cite{bingham1987large}, formula 1.5.8. page 26) we obtain that $\ell^*(x)$ is also a slowly varying function, satisfying 
\begin{equation}\label{karamata00}
\frac{\ell^*(n)}{\ell(n)}\to\infty.
\end{equation} 
For each positive integer $n$, let $b_n$ be the quantile of order $1-1/n$ of $F$, the distribution function of $\Delta_1$, i.e.
\begin{equation}\label{bn}
\frac{1}{1-F(b_n)}=n, \quad \forall n=1,2,3,\ldots. 
\end{equation}
The following Lemma establishes certain properties for the renewal process that will be useful to ensure certain uniform integrability conditions needed in the proof of Proposition \ref{PropSigmaY} as well as Theorem \ref{them1}.
\begin{lem}\label{technical}
Let $M_n=\max(\Delta_1,\ldots,\Delta_n)$, where $\Delta_j>0$ are i.i.d. (either continuous or integer-valued) random variables satisfying (\ref{heavy tail}).\\
(i) For any $r>0$ and $\alpha\le1$,
$$
\underset{n\ge1}{\sup}\,\mathbb{E}\left[\left(\frac{M_n+1}{b_n}\right)^{-r}\right]<\infty.
$$
(ii) For any $0<r<\alpha\le1$,
$$
\underset{n\ge1}{\sup}\,\int_{[0,1]^2}\mathbb{E}\left[\left(\frac{ T_{\vert[nx]-[ny]\vert}+1}{b_n}\right)^{-r}\right]dxdy<\infty.
$$
(iii) If $\alpha=1$ then 
\begin{equation*}
\underset{n\ge1}{\sup}\,\mathbb{E}\left[\left(\frac{T_n+1}{b_n}\frac{\ell(b_n)}{\ell^*(b_n)}\right)^{-2}\right]<\infty.
\end{equation*}
(iv) If $\alpha=1$ then for any $0<r<1$,
$$
\underset{n\ge1}{\sup}\,\int_{[0,1]^2}\mathbb{E}\left[\left(\frac{T_{\vert[nx]-[ny]\vert}+1}{b_n}\frac{\ell(b_n)}{\ell^*(b_n)}\right)^{-r}\right]dxdy<\infty.
$$
\end{lem} 
\begin{proof}
The proof is postponed to Appendix.
\end{proof}
\begin{prop}\label{PropSigmaY}
Assume that $X_t$ is a covariance-stationary process (of discrete or continuous time) with the covariance function (\ref{cov_stat}) and that $Y_k$ is defined by (\ref{sampled1}). If $\Delta_1$ satisfies (\ref{heavy tail}) and $\mathbb{E}(\Delta_1)=\infty$, then the asymptotic behaviour of the covariance function $\sigma_Y$ of $Y_k$ 
is given by
\begin{equation}\label{exact12}
\sigma_Y(h)\sim\begin{cases}
 
 \tilde C_d\frac{\Gamma\left(\frac{1-2d}{\alpha}\right)}{\alpha\Gamma(1-2d)}b_h^{2d-1},&\textrm{if }\alpha<1,\\\\
 \tilde C_db_h^{2d-1}\left(\frac{\ell^*(b_h)}{\ell(b_h)}\right)^{2d-1},&\textrm{if }\alpha=1,
 \end{cases}
\end{equation}
\end{prop}
\begin{proof}
Let first consider $0<\alpha<1$. Using (\ref{exact1}), we can write
\begin{eqnarray*}
\frac{\sigma_Y(h)}{b_h^{2d-1}}&=&\mathbb{E}\left(\frac{\sigma_X(T_h)}{b_h^{2d-1}}\right)=\mathbb{E}\left(u(T_h+1)\left(\frac{T_h+1}{b_h}\right)^{2d-1}\right):=\mathbb{E}(Z_h).
\end{eqnarray*}
Since $T_h\ge M_h$, 
$$
Z_h\le K\left(\frac{M_h+1}{b_h}\right)^{2d-1}.
$$ 
Thus $Z_h$ is uniformly integrable as shown in part (i) of Lemma \ref{technical}. Note that $T_h\overset{a.s.}{\to}\infty$ and hence $u(T_h+1)\overset{a.s.}{\to}\tilde C_d$. \\
Throughout this paper we use the functional central limit theorem for heavy tailed 
positive random variables when $0<\alpha<1$, \cite[see for example][Ch. 7, Corollary 7.1 and page 247] {resnick2007heavy}. 
\begin{equation}\label{fclt}
\frac{T_{[nt]}}{b_n}\overset{D[0,1]}{\to}L_t, \qquad\textrm{as }n\to\infty
\end{equation}
where $D[0,1]$ is the cadlag space of right continuous functions with finite left limit defined on [0,1], and $L_t$ is stable L\'evy subordinator process with parameter $\alpha$ (i.e., $L_t$ is nondecreasing in $t$). 
In particular, 
\begin{equation}\label{parti}
\frac{T_n}{b_n}\overset{\mathcal{D}}{\to}L_1,\qquad\textrm{as }n\to\infty. 
\end{equation}
and therefore 
$Z_h\overset{\mathcal{D}}{\to}\tilde C_dL_1^{2d-1}$. Using Theorem 25.12 of \cite{billingsley2017probability}, we then conclude that 
$$\mathbb{E}(Z_h)\to \tilde C_d\mathbb{E}(L_1^{2d-1})=\tilde C_d\frac{\Gamma\left(\frac{1-2d}{\alpha}\right)}{\alpha\Gamma(1-2d)},$$
using 
the fact that we have from \cite{Shanbhag1977OnCS}, for all $a>0$,
\begin{equation}\label{inverse_levy} 
\mathbb{E}\left(L_1^{-a}\right)=\frac{\Gamma(a/\alpha)}{\alpha\Gamma(a)}.
\end{equation}
Consider the case $\alpha=1$.
\begin{eqnarray*}
\mathbb{E}\left(\Delta_1\ind{\Delta\le n}\right)&=&\int_0^1xdF(x)+\int_1^nP(n\ge \Delta_1>x)dx=
\int_0^1xdF(x)+\ell^*(n)-\left(1-\frac{1}{n}\right)\ell(n)
\end{eqnarray*}
We know from \cite{resnick2007heavy}, page 219, formula 7.16, that 
$$
\frac{T_n}{b_n}-\frac{n}{b_n}\mathbb{E}\left(\Delta_1\ind{\Delta_1\le b_n}\right)\overset{\mathcal{D}}{\to}L_1.
$$ 
 By definition of $b_n$ in (\ref{bn}) we have $\frac{n}{b_n}=\frac{1}{\ell(b_n)}$. As $\mathbb{E}(\Delta_1)=\infty$, $\ell^*(n)\to\infty$ and hence using (\ref{karamata00}), we then obtain
\begin{equation}\label{harmonic}
\frac{T_n}{b_n}\frac{\ell(b_n)}{\ell^*(b_n)}\overset{P}{\to}1.
\end{equation}
We have 
$$
\frac{\sigma_Y(h)}{\left(b_h\frac{\ell^*(b_h)}{\ell(b_h)}\right)^{2d-1}}=
\mathbb{E}\left(\frac{\sigma_X(T_h)}{\left(b_h\frac{\ell^*(b_h)}{\ell(b_h)}\right)^{2d-1}}\right)
=\mathbb{E}\left(u(T_h+1)\left(\frac{T_h+1}{b_h\frac{\ell^*(b_h)}{\ell(b_h)}}\right)^{2d-1}\right):=\mathbb{E}(Z'_h).
$$
Using (iii) of Lemma \ref{technical}, we obtain that $Z'_h$ is uniformly integrable. This, combined with (\ref{harmonic}), implies that 
 $\mathbb{E}(Z'_h)\to\tilde C_d$. 
\end{proof}
\begin{rem}
 Equivalence (\ref{exact1}) provides exact asymptotic expressions for $\sigma_Y(h)$ as illustrated in the following discrete Pareto example:
$$ P(\Delta_1 = k) = \frac{1}{\zeta(1+\alpha)}k^{-1-\alpha},\qquad0<\alpha\le1,\quad k=1,2,3,\ldots,$$
 where $\zeta$ is the Riemann zeta function. Using the standard comparison between the sum and the integral of a decreasing function, we obtain that
 $$
 P(\Delta_1\ge n)=n^{-\alpha}\ell(n)
 $$
 where 
 $$
 \ell(n)\to\frac{1}{\alpha\zeta(1+\alpha)},\qquad\textrm{as }n\to\infty.
 $$
If the original process $X_n$ has covariance function of the form (\ref{cov_stat}), such as the well-known FARIMA processes, then
\cite{philippe-viano2008} obtained 
\begin{enumerate}
\item\label{item:1} if $\alpha > 1$,
\begin{equation}\label{less1}
 \sigma_Y(h)\sim C h^{2d-1}.
\end{equation}
In this case $X$ and $Y$ have the same memory parameter $d$.
\item \label{item:2} if $\alpha \leq 1$,
\begin{equation}\label{less2}
 C_1 h^{(2d-1)/(\alpha) -\epsilon} \leq
 \sigma_Y(h)\leq C_2 h^{(2d-1)/(\alpha)},\quad \forall \epsilon>0.
\end{equation}
\end{enumerate}
In the above, $C,C_1,C_2$ are some positive constants.
In this second case, the lower and upper bounds do not allow us to know the long memory parameter of the sampled process $Y_n$. 
Thanks to proposition \ref{PropSigmaY}, we obtain that the memory parameter is equal to $\dfrac{2d+\alpha - 1}{2 \alpha}$, which is smaller than $d$ when $\alpha<1$ and equal to $d$ when $\alpha=1$. Indeed, as $h\to\infty$,

we have
$$
b_h\sim
\frac{1}{\alpha\zeta(1+\alpha)} h\frac{1}{\alpha}=\frac{1}{\zeta(2)}=\frac{6}{\pi^2}h
$$
and, therefore, 
\begin{equation*}
\sigma_Y(h)\sim\begin{cases}
 \tilde C_d\frac{\alpha\Gamma\left(\frac{1-2d}{\alpha}\right)}{(1-2d)\Gamma(1-2d)}\left(\frac{1}{\alpha\zeta(1+\alpha)}\right)^{2d-1}h^{\frac{2d-1}{\alpha}},&\textrm{if }\alpha<1,\\\\
 \tilde C_d\left(\frac{6}{\pi^2}\right)^{2d-1}h^{2d-1}(\log h)^{2d-1}&\textrm{if }\alpha=1.
 \end{cases}
\end{equation*}
Note that even when $\alpha=1$, $\sigma_Y(h)$ converges to zero faster than $\sigma_X(h)$ due to the presence of $\log h$ in $\sigma_Y(h)$.
\end{rem}
\begin{rem}
Convergence (\ref{harmonic}), combined with (iii) of Lemma \ref{technical}, implies that
\begin{equation}\label{log_lim}
\mathbb{E}\left[\left(\frac{T_n}{b_n}\frac{\ell(b_n)}{\ell^*(b_n)}\right)^{-1}\right]\to1.
\end{equation}
\cite{PMID:25331886} showed that if $U_1,\ldots,U_n$ are i.i.d. copies from the uniform (0,1) distribution then
\begin{equation}\label{old1}
(\ln n)\mathbb{E}\left(\frac{n}{\frac{1}{U_1}+\cdots+\frac{1}{U_n}}\right)\to1.
\end{equation}
We retrieve this asymptotic expression of the expected value of the harmonic mean of i.i.d. uniformly distributed random variables as a particular case of (\ref{log_lim}) by considering continuous $\Delta_j$ with support $[1,\infty)$ and 
$$
P(\Delta_j>x)=\frac{1}{x},\qquad x>1.
$$
We will have 
 $\ell(n)=1$, $b_n=n$, and from (\ref{elstar}),
$$
\ell^*(n)=\int_1^n\frac{1}{x}dx=\ln n,
$$
and $T_n=\Delta_1+\cdots+\Delta_n$ has the same distribution as 
$$
\frac{1}{U_1}+\cdots+\frac{1}{U_n}.
$$
\end{rem}\noindent
\section{Asymptotic behaviour of the sums of the randomized process}
Assume that originally we have a linear process
\begin{equation}\label{lin11}
X_t=\sum_{i=0}^\infty a_i\epsilon_{t-i}=\sum_{j=-\infty}^ta_{t-j}\epsilon_j,
\end{equation}
where $\epsilon_j$ are i.i.d. with mean zero and variance $\sigma^2_\epsilon<\infty$ and
$$
\sum_{j=1}^\infty a_j^2<\infty.
$$
In the sequel we consider a linear long memory process by assuming that
\begin{equation}\label{eq:a}
a_i \sim C_d i^{d-1} \text{as } i\to \infty,
\end{equation}
where $0<d<1/2$ and $C_d$ is a positive constant. Condition (\ref{eq:a}) ensures that $X_t$ is of long memory with covariance function satisfying condition (\ref{cov_stat}) with
\begin{equation}\label{new C}
\tilde C_d:=\sigma_\epsilon^2C_d^2\,\,\beta(d,1-2d),
\end{equation}
where $\beta$ is the beta function (see e.g. \cite{surgailis2012large}, Proposition 3.2.1.).
We study the asymptotic behaviour of the sum of the sampled process $Y_k$ defined in (\ref{sampled1}), 
$$
\sum_{k=1}^nY_k=\sum_{k=1}^nX_{T_k}
$$
Firstly, we establish in the following proposition the asymptotic normality of its partial sums when normalized by the conditional variance, irrespective of $\mathbb{E}(\Delta_1)$ being finite or not. 
\begin{lem}\label{lem1}
 (i) Let $\sigma_X$ and $\sigma_Y$ be the covariance functions of the stationary processes $X_t$ and $Y_t$ defined in (\ref{lin11})-(\ref{eq:a}) and (\ref{sampled1}). 
 There exists $m\geq 1 $ such that
$\sigma_X(h) >0$ for all $h\ge m-1 $ and 
$$ \sum_{j= -m } ^m \sigma_Y(j) >0. $$
(ii) Let \begin{equation}\label{dn}
d_{n,j}=\sum_{k=1}^n a_{T_k-j},\qquad d_n^2 := \sum_{j\in\mathbb{Z}}d^2_{n,j}.
\end{equation}
\begin{equation}\label{lindeberg}
\sup_{j\in\mathbb{Z}}\frac{d^2_{n,j}}{\sum_{j\in\mathbb{Z}}d^2_{n,j}}\overset{P}{\to}0.
\end{equation} 
\end{lem}
\begin{proof}
The proof is postponed to Appendix.
\end{proof}\noindent
Denote $\textrm{Var}(.\vert T)$ and $\mathbb{E}(.\vert T)$ respectively the conditional variance and expectation given $\Delta_1,\Delta_2,\ldots$.
\begin{prop}\label{propo_clt} Let $X_t$ be the linear process defined in (\ref{lin11})-(\ref{eq:a}) and let $Y_t$ be the sampled process defined in (\ref{sampled1}). 
We have, as $n\to\infty$, 
\begin{equation}\label{conditional1}
S'_n ( X, T):=\left[\textrm{Var}\left(\sum_{k=1}^nY_k\vert T\right)\right]^{-1/2}\left(\sum_{k=1}^nY_k\right)\overset{\mathcal{D}}{\to}\mathcal{N}(0,1).
\end{equation}
\end{prop}\noindent
\begin{proof}
 First, putting $a_u=0$ if $u<0$, we can write
\begin{eqnarray*}
\sum_{k=1}^nY_k&=&\sum_{k=1}^nX_{T_k}=\sum_{k=1}^n\sum_{s=1}^\infty X_s\ind{T_k=s}=\sum_{s=1}^\infty X_s\left(\sum_{k=1}^n\ind{T_k=s}\right)\\
&=&\sum_{s=1}^\infty\left(\sum_{j=-\infty}^sa_{s-j}\epsilon_j\right)\left(\sum_{k=1}^n\ind{T_k=s}\right)
=\sum_{j=-\infty}^\infty\left(\sum_{s=1}^{T_n} a_{s-j}\left(\sum_{k=1}^n\ind{T_k=s}\right)\right)\epsilon_j\\
&&=\sum_{j=-\infty}^\infty\left(\sum_{k=1}^n a_{T_k-j}\right)\epsilon_j
:=\sum_{j=-\infty}^\infty d_{n,j}\epsilon_j.
\end{eqnarray*}
Using the independence of the sequence $(d_{n,j})_{n,j}$ with the sequence $ (\epsilon_j)_{j}$ of iid random variables and the well-known fact that if $X,Y$ are independent then $\mathbb{E}(f(X,Y)\vert Y)=g(Y)$ where $g(y)=\mathbb{E}(f(X,y))$, we have for almost every $\omega$
$$
\E\left( e^{it S'_n(X,T)} \vert T \right)(\omega)=\E\left(\exp\left(it\sum_{j=-\infty}^\infty\frac{d_{n,j}(\omega)}{\sigma_\epsilon d_n(\omega)}\epsilon_j\right)\right),
$$
where $d_n^2 := \sum_{j\in\mathbb{Z}}d^2_{n,j}$. 
From the relationship between convergence in probability and convergence almost sure, we know that (ii) of Lemma \ref{lem1} is equivalent to the fact that every subsequence $n'$ of $n$ contains a further subsequence $n''$ such that as $n''\to\infty$,
\begin{equation*}
\sup_{j\in\mathbb{Z}}\frac{d^2_{n'',j}}{\sum_{j\in\mathbb{Z}}d^2_{n'',j}}\overset{\textrm{a.s.}}{\longrightarrow}0.
\end{equation*} 
Hence, using Corollary 4.3.1. of \cite{surgailis2012large} (stated for deterministic coefficients $d_{nj}$) and (ii) of Lemma \ref{lem1} above, we obtain that every subsequence $n'$ contains a further subsequence $n''$ such that for almost every $\omega$,
\begin{equation}\label{almost sure}
\sum_{j=-\infty}^\infty\frac{d_{n'',j}(\omega)}{\sigma_\epsilon d_n''(\omega)}\epsilon_j\overset{\mathcal{D}}{\to}\mathcal{N}(0,1),
\end{equation}
which in terms of characteristic functions is equivalent to the fact that for every $t\in\mathbb{R}$,
$$ 
\E\left( e^{it S'_{n''}(X,T)} \vert T \right)\overset{\textrm{a.s.}}{\longrightarrow}e^{-\frac{t^2}{2}}.
$$
 or equivalently, 
for every $t\in\mathbb{R}$,
$$ 
\E\left( e^{it S'_n(X,T)} \vert T \right)\overset{\textrm{P}}{\longrightarrow}e^{-\frac{t^2}{2}},
$$
which implies
$$
\E\left(e^{itS'_n(X,T)}\right)\to e^{-\frac{t^2}{2}},
$$
by the Bounded Convergence Theorem,
or equivalently
$$
S_n'(X,T)\overset{\mathcal{D}}{\to}\mathcal{N}(0.1),
$$
which completes the proof of (\ref{conditional1}).
\end{proof}
The following Theorem shows the difference that might exist between the asymptotic behaviour of the sums of the orignal process $X_k$ and the sampled one $Y_k$. It is well known that for linear processes defined in (\ref{lin11}) (see \cite{ibragimov1971independent}, Theorem 18.6.5)
$$
\left[\textrm{Var}\left(\sum_{k=1}^nX_k\right)\right]^{-1/2}\left(\sum_{k=1}^nX_k\right)\overset{\mathcal{D}}{\to}\mathcal{N}(0,1).
$$
\begin{thm}\label{them1}
Assume $X_t$ is a linear process as defined in (\ref{lin11}) and (\ref{eq:a}). 
Let
\begin{equation}\label{clt}
S_n ( X, T) = 
\left[\textrm{Var}\left(\sum_{k=1}^nY_k\right)\right]^{-1/2}\left(\sum_{k=1}^nY_k\right).
\end{equation} \\
(i) if $\mathbb{E}(\Delta_1)<\infty$ or if $P(\Delta_1>k)=k^{-\alpha}\ell(k)$, where $\ell$ is a slowly varying function, and $0<\alpha\le1-2d$ or $\alpha=1$, then as $n\to\infty$,
\begin{equation}\label{clt1}
S_n ( X, T)\overset{\mathcal{D}}{\to}\mathcal{N}(0,1),
\end{equation}
(ii) if $P(\Delta_1>k)=k^{-\alpha}\ell(k)$ such that $1-2d<\alpha<1$ and $\ell$ is a slowly varying function, then as $n\to\infty$,
\begin{equation}\label{clt11}
S_n ( X, T)\overset{\mathcal{D}}{\to}\sqrt{Z(\alpha,d)}N,
\end{equation}
where $N$ has standard normal distribution and
\begin{equation}\label{zed}
Z(\alpha,d)=C_{\alpha,1-2d}
\int_{[01]^2}\vert L_x-L_y\vert^{2d-1}dxdy,
\end{equation}
where $L_t$ is a L\'evy stable motion with parameter $\alpha$, $L_0=0$ and $L_t$ is nondecreasing in $t$ (i.e., $L_t$ is a stable subordinator), which is independent of the Gaussian variable $N$,
and
\begin{equation}\label{down1}
C_{\alpha,1-2d}=
\frac{(\alpha+2d-1)(2\alpha+2d-1)}{2\alpha^2}\frac{1}{\mathbb{E}(L_1^{2d-1})}=
(\alpha+2d-1)(2\alpha+2d-1)\frac{\Gamma(1-2d)}{2\alpha\Gamma\left(\frac{1-2d}{\alpha}\right)}.
\end{equation}
(iii) The distribution of $Z(\alpha,d)$ is determined by its moments and 
\begin{equation}\label{continuity}
Z(\alpha,d)\overset{P}{\to}1,\qquad\textrm{as }\alpha\to1-2d\textrm{ or }\alpha\to1.
\end{equation}
\end{thm}\vspace{1cm}\noindent
A graph summary of the limiting distribution's nature as a function of $d$ and $\alpha$. (The area $\alpha>1$ denotes the class of renewal processes with finite moment, without being necessarily of heavy tail).\\

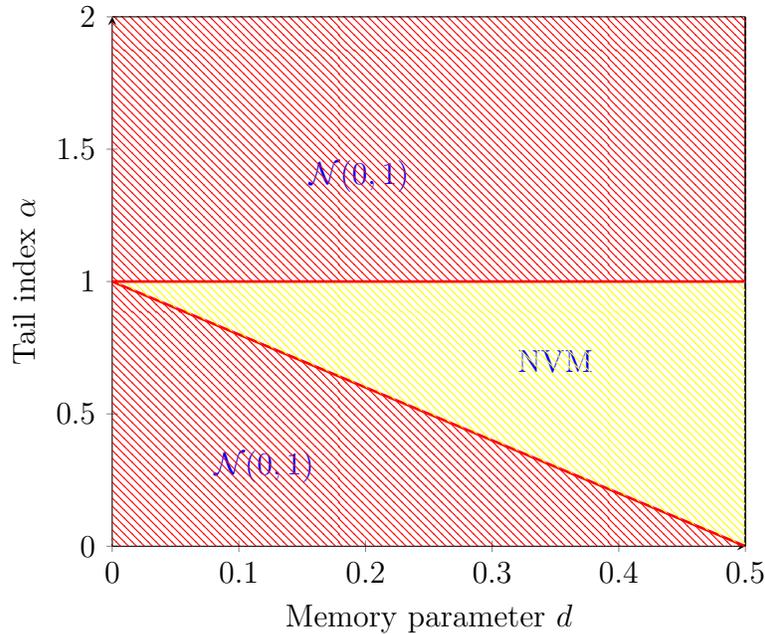
\begin{center}
\begin{tikzpicture}
\begin{axis}[
 axis lines = left,
 xlabel = Memory parameter \(\Large d\),
 ylabel =Tail index \(\alpha\),
]
\draw[thick,gray] (0,1) -- (0,2);
\addplot [
 domain=0:.5, 
 samples=100, 
 color=red,
 line width=1pt
] 
{1-2*x};
\node[blue] at (350, 70) {NVM};
\node[blue] at (120, 30) {$\mathcal{N}(0,1)$};
\node[blue] at (195,140) {$\mathcal{N}(0,1)$};
\addplot[thick,smooth] coordinates {(.5,0)(.5,2)};
\fill[pattern=north west lines, pattern color=yellow] (axis cs: 0,1) -- (axis cs: 0.5,0) -- (axis cs: 0.5,1) -- cycle;
\fill[pattern=north west lines, pattern color=red] (axis cs: 0,0) -- (axis cs: 0.5,0) -- (axis cs: 0,1) -- cycle;
\fill[pattern=north west lines, pattern color=red] (axis cs: 0,1) -- (axis cs: 0.5,1) -- (axis cs: 0.5,2) -- (axis cs: 0,2) -- cycle;
\addplot [
 domain=0:1/2, 
 samples=100, 
 color=red,
 line width=1pt
]{1};
\addplot [
 domain=0:1/2, 
 samples=100, 
 color=black,
]{2};
\end{axis}
\end{tikzpicture}
\end{center}
{\bf Proof of Theorem \ref{them1}}. It is organized as follows.
Using Proposition \ref{propo_clt}, to establish (i) and (ii), we will prove\\
(i) If $\mathbb{E}(\Delta_1)<\infty$ or $P(\Delta_1>k)=k^{-\alpha}\ell(k)$ where $0<\alpha\le 1-2d$ or if $\alpha=1$, then 
\begin{equation}\label{conditional2}
R_n(T):=\frac{\textrm{Var}\left(\sum_{k=1}^nY_k\vert T\right)}{
\textrm{Var}\left(\sum_{k=1}^nY_k\right)}
\overset{P}{\to}1.
\end{equation}
To get this result, the techniques are different depending on the assumption on $\Delta_1$.\\
(ii) If $P(\Delta_1>k)\sim k^{-\alpha}$, with $1-2d<\alpha<1$, then 
\begin{equation}\label{conditional3}
(R_n(T),S'_n ( X, T))\overset{\mathcal{D}}{\to}(Z(d,\alpha),N),
\end{equation}
since $S_n(X,T)=\sqrt{R_n(T)}S'_n(X,T).$\\
(iii) We will show that
\begin{equation}\label{moments_problem}
\sum_{k=1}^\infty\left(\mathbb{E}\left(Z(\alpha,d))^k\right)\right)^{\frac{-1}{k}}=\infty,
\end{equation}
to guarantee that $Z(\alpha,d)$'s distribution is determined by its moments
and prove (\ref{continuity}). 
\\\\
In what follows, $C$ will denote a generic positive constant that may change from one expression to another.\\
{\bf Proof of (\ref{conditional2}) when $\bm{\mathbb{E}(\Delta_1)<\infty}$.} 
As already mentioned, the linear processes with coefficients defined in (\ref{eq:a}) have a covariance function satisfying (\ref{exact1}), that is 
\begin{equation*}
\sigma_X(h)=u(h+1)(h+1)^{2d-1},
\end{equation*} 
with $u(h)\to \tilde C_d$ as $h\to\infty$ where $\tilde C_d$ is defined in (\ref{new C}) and $\vert u\vert$ is bounded. Note that for $x>y$, we have $(T_{[nx]}-T_{[ny]}+1)/n\ge([nx]-[ny]+1)/n\ge(x-y)$ and hence for all $x\neq y$, 
$$
\frac{\vert\sigma_X(T_{[nx]}-T_{[ny]})\vert}{n^{2d-1}}=\vert u(\vert T_{[nx]}-T_{[ny]}\vert+1)\vert
\left(\frac{\vert T_{[nx]}-T_{[ny]}\vert+1}{n}\right)^{2d-1}
\le C\vert x-y\vert^{2d-1}
$$
and 
\begin{equation}\label{value in d}
\int_{[0,1]^2}\vert x-y\vert^{2d-1}dxdy=\frac{1}{d(2d+1)}\in(0,\infty),
\end{equation}
and by the Law of Large Numbers, for all $x\neq y$,
$$
\left(\frac{\vert T_{[nx]}-T_{[ny]}\vert+1}{n}\right)^{2d-1}\overset{a.s.}{\to}\left(\mathbb{E}(\Delta_1)\right)^{2d-1}
\vert x-y\vert^{2d-1}>0.
$$
Then applying Lebesgue Dominated Convergence Theorem to both the numerator and the denominator below, we get
\begin{eqnarray*}
R_n(T)&=&\frac{n^2\int_{[0,1]^2}\sigma_X(T_{[nx]}-T_{[ny]})dxdy}
{n^2\int_{[0,1]^2}\mathbb{E}\left(\sigma_X(T_{[nx]}-T_{[ny]})\right)dxdy}
=\frac{\int_{[0,1]^2}\frac{\sigma_X(T_{[nx]}-T_{[ny]})}{n^{2d-1}}dxdy}
{\int_{[0,1]^2}\mathbb{E}\left(\frac{\sigma_X(T_{[nx]}-T_{[ny]})}{n^{2d-1}}\right)dxdy}\\
&=&\frac{\int_{[0,1]^2}u(\vert T_{[nx]}-T_{[ny]}\vert+1)\left(\frac{\vert T_{[nx]}-T_{[ny]}\vert+1}{n}\right)^{2d-1}dxdy}
{\int_{[0,1]^2}\mathbb{E}\left[u(\vert T_{[nx]}-T_{[ny]}\vert+1)\left(\frac{\vert T_{[nx]}-T_{[ny]}\vert+1}{n}\right)^{2d-1}\right]dxdy}\overset{a.s.}{\to}\frac{\left(\mathbb{E}(\Delta_1)\right)^{2d-1}/(d(2d+1))}{\left(\mathbb{E}(\Delta_1)\right)^{2d-1}/(d(2d+1))}=1,
\end{eqnarray*}
which completes the proof of (\ref{conditional2}) under the finite moment assumption
$\mathbb{E}(\Delta_1)<\infty$.\\\\
{\bf Proof of (\ref{conditional2}) when $\bm{P(\Delta_1>k)=k^{-\alpha}\ell(k)}$ and $\bm{\alpha<1-2d}$ or ($\bm{\alpha=1-2d}$ and $\bm{b_n^{-\alpha}}$ summable)}.
An example of the last case is when $\ell(n)=(\ln(e-1+n))^2$, since from (\ref{bn}), $b_n^{-\alpha}=\left(n\ell(b_n)\right)^{-1}$ and $b_n>n$ and hence $b_n^{-\alpha}\le n^{-1}(\ln(e-1+n))^{-2}$ whose sum is clearly bounded by $e$.
Note that $b_n$ is the inverse function of a nondecreasing regularly varying function with coefficient $\alpha$ and hence $b_n$ is of the form $b_n=n^{1/\alpha}z(n)$ where $z$ is a slowly varying function (at infinity) and hence for $\alpha<1-2d$, $b_n^{2d-1}$ will also be summable. Since for $n\ge2$, $T_n\ge M_n+1$, we conclude from (i) of Lemma \ref{technical} that
$\mathbb{E}(T_n^{2d-1})=O(b_n^{2d-1})$ and 
hence $\mathbb{E}(T_n^{2d-1})$ is summable. That is, in both cases, when $\alpha<1-2d$ or when $\alpha=1-2d$ but $b_n^{-\alpha}$ remains summable, we obtain 
 \begin{equation}\label{defined}
 \sum_{h=1}^\infty\mathbb{E}(\vert\sigma_X(T_h))\vert)<C\sum_{h=1}^\infty\mathbb{E}(T_h^{2d-1})<\infty.
\end{equation}
\begin{eqnarray*}
\frac{1}{n}\textrm{Var}\left(\sum_{k=1}^nY_k\right)&=&\frac{1}{n}\sum_{k=1}^n\sum_{k'=1}^n\mathbb{E}\left(\sigma_X(T_k-T_{k'})\right)\\
&=&\sigma_X(0)+2\sum_{h=1}^n\left(1-\frac{h}{n}\right)\mathbb{E}(\sigma_X(T_h))
\to \sigma_X(0)+2\sum_{h=1}^\infty\mathbb{E}(\sigma_X(T_h)).
\end{eqnarray*}
We have
$$
\frac{1}{n}\sum_{k=1}^n\sum_{k'=1}^n\sigma_X(T_k-T_{k'})=\sigma_X(0)+
\frac{1}{n}\sum_{k=1}^n\left(2\sum_{h=1}^{n-k}\sigma_X(\Delta_{k+1}+\cdots+\Delta_{k+h})\right).
$$ 
For all $n$ and $k\le n$, let
$$
Z_{nk}:=\sum_{h=1}^{n-k}\sigma_X(\Delta_{k+1}+\cdots+\Delta_{k+h})
$$
and for all $k\ge1$, let
$$
Z_k:=\sum_{h=1}^\infty\sigma_X(\Delta_{k+1}+\cdots+\Delta_{k+h})
$$
which is well defined in $L^1$ according to (\ref{defined}), and
\begin{equation}\label{grosse}
\mathbb{E}(Z_k)=\mathbb{E}(Z_1)=\sum_{h=1}^\infty\mathbb{E}(\sigma_X(T_h)).
\end{equation}
Also, combining Cesaro Lemma and the fact that the remainder term of a converging series converges to zero, we get
$$
\frac{1}{n}\sum_{k=1}^n\mathbb{E}\left(\vert Z_k-Z_{nk}\vert\right)\le
\frac{1}{n}\sum_{k=1}^n\sum_{h=k}^\infty\mathbb{E}\left(\vert\sigma_X(T_h)\vert\right)\to0
$$
and hence
\begin{equation}\label{inL1}
\left(
\frac{1}{n}\sum_{k=1}^nZ_{nk}\right)-\left(\frac{1}{n}\sum_{k=1}^nZ_k\right)\overset{L^1}{\to}0.
\end{equation}
Note also in passing that the process $Z=(Z_k)_{k\ge1}$ is stationary and ergodic (see \cite{billingsley2017probability}, Theorem 36.4) as it is obtained as $Z_k=\Psi\circ B^k(\Delta)$ where $\Delta=(\Delta_k)_{k\ge1}$ is an i.i.d. process, $B$ is the backshift operator and
$$
\Psi:[1,\infty)^\mathbb{N}\to\mathbb{R},\qquad \Psi(x)=\sum_{h=1}^\infty\sigma_X(x_1+\cdots +x_h)
$$
is clearly measurable. Therefore by The Erogodic Theorem we have
$$
\frac{1}{n}\sum_{k=1}^nZ_k\overset{a.s}{\to}\mathbb{E}(Z_1).
$$
which implies, using (\ref{inL1}) and (\ref{grosse}) that
$$
\frac{1}{n}\sum_{k=1}^nZ_{nk}\overset{P}{\to}\sum_{h=1}^\infty\mathbb{E}(\sigma_X(T_h)).
$$
 This completes the proof of (\ref{conditional2}) under the extreme heavy tail assumption $P(T_1>k)=k^{-\alpha}\ell(k)$, $0<\alpha<1-2d$ or when $\alpha=1-2d$ but $b_n^{-\alpha}$ remains summable. \\ 
{\bf Proof of (\ref{conditional2}) when $\bm{P(\Delta_1>k)=k^{-\alpha}\ell(k)}$, $\bm{\alpha=1-2d}$ and $\bm{b_n^{-\alpha}}$ is not summable.} \\
Rewrite $b_h^{-\alpha}$ as 
$$
b_h^{-\alpha}=\frac{1}{h\ell(b_h)}=\frac{1}{h+1}\frac{h+1}{h}\frac{1}{\ell(b_h)}:=\frac{g(h)}{h+1}, 
$$
where $g(h)$ is a positive locally bounded slowly varying function at infinity.
According to Proposition \ref{PropSigmaY}, 
we obtain that 
\begin{equation}\label{asymp1}
 \mathbb{E}(\sigma_X(T_h))=\sigma_Y(h)=b_h^{-\alpha}g_1(h)=\frac{g(h)g_1(h)}{h+1},
 \end{equation}
 and $g_1(h)\to\tilde C_d\mathbb{E}(L_1^{-\alpha})$ as $h\to\infty$. \\
Also, clearly from (i) of Lemma \ref{technical}, taking $r=2(1-2d)$, we have
 \begin{equation}
 \left(\mathbb{E}(\sigma^2_X(T_h))\right)^{\frac{1}{2}}\le C\frac{g(h)}{h+1},
 \end{equation}
Therefore, with
 $$
 V_n=\sum_{h=1}^nb_h^{-\alpha}=\sum_{h=1}^n\frac{g(h)}{h+1},
 $$
 $V_n\to\infty$ by assumption, and increasing (for $n\ge h_0$, for certain $h_0>0$). We get 
 \begin{eqnarray*}\lefteqn{
 \frac{1}{nV_n}\sum_{k=1}^n\sum_{k'=1}^n\mathbb{E}\left(\sigma_X(T_k-T_{k'})\right)=\frac{\sigma_X(0)}{V_n}+\frac{2}{V_n}\sum_{h=1}^n\left(1-\frac{h}{n}\right)\mathbb{E}(\sigma_X(T_h))}\\
 &&=
 \frac{\sigma_X(0)}{V_n}+\frac{2}{V_n}\sum_{h=1}^n\left(1-\frac{h}{n}\right)b_h^{-\alpha}g_1(h)\\
 &&\to2\tilde C_d\mathbb{E}(L_1^{2d-1})>0,\qquad\textrm{as }n\to\infty.
 \end{eqnarray*}
 Also, using the properties of slowly varying functions, we obtain by Karamata Theorem that $g(n)/V_n\to0$ as $n\to\infty$ and since $V_n\to\infty$ (and increasing for $n\ge h_0$), we can easily check that 
 \begin{equation}\label{maxout}
 \frac{\underset{h\le n}{\max}\,\,g(h)}{V_n}\to0,\qquad\textrm{as }n\to\infty.
 \end{equation}
 Let
 $$
 Z'_{n,k}=\frac{1}{V_n}\sum_{h=1}^{n-k}\sigma_X(\Delta_{k+1}+\cdots+\Delta_{k+h}).
 $$
 In order to show that (\ref{conditional2}) still holds when $\alpha=1-2d$ and $V_n\to\infty$, it will be enough to show that, as $n\to\infty$,
 \begin{equation}\label{uniform_kk}
 \underset{k\le n}{\max}\textrm{Var}(Z'_{n,k})\to0.
 \end{equation} 
 Indeed, this will clearly imply that
 $$
 \underset{k\le n}{\max}\,\mathbb{E}\left\vert Z_{n,k}'-\mathbb{E}(Z_{n,k}')\right\vert\to0
 $$
and then we have
 \begin{eqnarray*}\lefteqn{ 
 \mathbb{E}\left\vert\left(\frac{1}{nV_n}\sum_{k=1}^n\sum_{k'=1}^n\mathbb{E}\left(\sigma_X(T_k-T_{k'})\right)\right)^{-1}
 \frac{1}{nV_n}\sum_{k=1}^n\sum_{k'=1}^n\sigma_X(T_k-T_{k'})-1\right\vert}\\
 &&=\left(\frac{1}{nV_n}\sum_{k=1}^n\sum_{k'=1}^n\mathbb{E}\left(\sigma_X(T_k-T_{k'})\right)\right)^{-1}
 \frac{1}{nV_n}\mathbb{E}\left\vert\left[\sum_{k=1}^n\sum_{k'=1}^n\sigma_X(T_k-T_{k'})-\sum_{k=1}^n\sum_{k'=1}^n\mathbb{E}\left(\sigma_X(T_k-T_{k'})\right)\right]\right\vert\\
 &&=
 \left(\frac{1}{nV_n}\sum_{k=1}^n\sum_{k'=1}^n\mathbb{E}\left(\sigma_X(T_k-T_{k'})\right)\right)^{-1}\frac{1}{n}\mathbb{E}\left\vert\left[\sum_{k=1}^n\left(Z_{n,k}'-\mathbb{E}(Z_{n,k}'\right)\right]\right\vert\\
 &&\le\left(\frac{1}{nV_n}\sum_{k=1}^n\sum_{k'=1}^n\mathbb{E}\left(\sigma_X(T_k-T_{k'})\right)\right)^{-1}\frac{1}{n}\sum_{k=1}^n\mathbb{E}\left\vert Z_{n,k}'-\mathbb{E}(Z_{n,k}')\right\vert\to0.
 \end{eqnarray*}
 Let us now prove (\ref{uniform_kk}). 
 We also note that $\vert\sigma_X(h)\vert\le\sigma_X(0)$, and for all $h\ge h_0$, $\sigma_X(T_h)\ge0$, and 
 where $g$ satisfies (\ref{maxout}).
 Let $\epsilon>0$.
We can write for $n$ such that $\sqrt{V_n}>m$ (where $m$ is as in (i) of Lemma \ref{lem1}),
\begin{eqnarray*}\lefteqn{
\textrm{Var}(Z_{n,k}')=\frac{n^2}{(V_n)^2}\int_0^{1-\frac{k}{n}}\int_0^{1-\frac{k}{n}}\left[\mathbb{E}\left(\sigma_X(T_{[ns]})\sigma_X(T_{[nt]})\right)-\mathbb{E}\left((\sigma_X(T_{[ns]})\right)\mathbb{E}\left(\sigma_X(T_{[nt]})\right)\right]dsdt}\\
&&=
\frac{n^2}{(V_n)^2}\int_0^{\frac{(V_n)^{\frac{1}{2}}}{n}}\int_0^{\frac{(V_n)^{\frac{1}{2}}}{n}}\left[\mathbb{E}\left(\sigma_X(T_{[ns]})\sigma_X(T_{[nt]})\right)-\mathbb{E}\left((\sigma_X(T_{[ns]})\right)\mathbb{E}\left(\sigma_X(T_{[nt]})\right)\right]dsdt\\
&&+\frac{2n^2}{(V_n)^2}\int_{\left\{\frac{(V_n)^{\frac{1}{2}}}{n}<s<t<1-\frac{k}{n}\right\}}\left[\mathbb{E}\left(\sigma_X(T_{[ns]})\sigma_X(T_{[nt]})\right)-\mathbb{E}\left((\sigma_X(T_{[ns]})\right)\mathbb{E}\left(\sigma_X(T_{[nt]})\right)\right]dsdt\\
&&\le\frac{2\sigma^2_X(0)}{V_n}+ 
\frac{2n^2}{(V_n)^2}\int_\frac{(V_n)^{\frac{1}{2}}}{n}^{1-\frac{k}{n}}\int_\frac{(V_n)^{\frac{1}{2}}}{n}^{\epsilon t}\left[\mathbb{E}\left(\sigma_X(T_{[ns]})\sigma_X(T_{[nt]})\right)-\mathbb{E}\left((\sigma_X(T_{[ns]})\right)\mathbb{E}\left(\sigma_X(T_{[nt]})\right)\right]dsdt\\
&&+
\frac{2n^2}{(V_n)^2}\int_\frac{(V_n)^{\frac{1}{2}}}{n}^{1-\frac{k}{n}}\int_{\epsilon t}^t\left(\mathbb{E}(\sigma^2_X(T_{[ns]}))\mathbb{E}(\sigma^2_X(T_{[nt]}))\right)^{\frac{1}{2}}dsdt
:=\frac{2\sigma^2_X(0)}{V_n}+I(\epsilon,n)+II(\epsilon,n).
\end{eqnarray*}
Noting that 
$$
\frac{1}{1+nu}\le\frac{1}{1+[nu]}\le\frac{1}{nu}, 
$$
we obtain that for every (small) $\epsilon>0$,
\begin{eqnarray*}
II(\epsilon,n)&=&\frac{1}{(V_n)^2}\int_\frac{(V_n)^{\frac{1}{2}}}{n}^{1-\frac{k}{n}}n\left(\int_{\epsilon t}^t\frac{n}{1+[ns]}g([ns])g_1([ns])ds\right)\left(\mathbb{E}(\sigma^2_X(T_{[nt]}))\right)^{\frac{1}{2}}dt\\
&\le&C\left(\frac{\underset{h\le n}{\max}\,\,g(h)}{V_n}\right)\frac{1}{V_n}\int_\frac{1}{n}^1n\left(\int_{\epsilon t}^t\frac{1}{s}ds\right)\left(\mathbb{E}(\sigma^2_X(T_{[nt]}))\right)^{\frac{1}{2}}dt\\
&=&C\left(\frac{\underset{h\le n}{\max}\,\,g(h)}{V_n}\right)(-\log\epsilon)\frac{n}{V_n}
\left(\int_{\frac{1}{n}}^1\left(\mathbb{E}(\sigma^2_X(T_{[nt]}))\right)^{\frac{1}{2}}dt\right)\\
&\le&\left(-\log\epsilon\right)C\left(\frac{\underset{h\le n}{\max}\,\,g(h)}{V_n}\right)\to0,
\end{eqnarray*}
as $n\to\infty.$\\
Since $T_h\ge h$ for all $h\ge1$, we have for 
$$t\in\left[\frac{(V_n)^{\frac{1}{2}}}{n},1-\frac{k}{n}\right],
$$
$$
\underset{j\ge\sqrt{V_n}}{\min}u(j)\le u(T_{[nt]})\le\underset{j\ge\sqrt{V_n}}{\max}u(j),
$$
and hence, since $\sigma_X(h)=u(h+1)^{-\alpha} u(h+1)$,
$$
\underset{j\ge\sqrt{V_n}}{\min}u(j)T_{[nt]}^{-\alpha}\le\sigma_X(T_{[nt]})\le \underset{j\ge\sqrt{V_n}}{\max}u(j)T_{[nt]}^{-\alpha}\le\underset{j\ge\sqrt{V_n}}{\max}u(j)\left(T_{[nt]}-T_{[n\epsilon t]}\right)^{-\alpha}.
$$
Using the stationarity and independence of the increments of $T_h$, and the fact that for $x>y$, $[x-y]\le[x]-[y]$, we obtain for $0<\epsilon<1/2$
\begin{eqnarray*}\lefteqn{
I(\epsilon, n)\le}\\
&&\frac{2n^2}{(V_n)^2}\int_\frac{(V_n)^{\frac{1}{2}}}{n}^{1-\frac{k}{n}}
\int_\frac{(V_n)^2}{n}^{\epsilon t}
\left[\underset{j\ge\sqrt{V_n}}{\max}u(j)\mathbb{E}\left(\sigma_X(T_{[ns]})T^{-\alpha}_{[nt(1-\epsilon)]}\right)-\mathbb{E}\left(\sigma_X(T_{[ns]})\right)\underset{j\ge\sqrt{V_n}}{\min}u(j)\mathbb{E}\left(T_{[nt]}^{-\alpha}\right)\right]dsdt\\
&&=
\frac{2n^2}{(V_n)^2}\int_\frac{(V_n)^{\frac{1}{2}}}{n}^{1-\frac{k}{n}}
\int_\frac{(V_n)^2}{n}^{\epsilon t}
\mathbb{E}\left(\sigma_X(T_{[ns]})\right)ds
\left[\underset{j\ge\sqrt{V_n}}{\max}u(j)\mathbb{E}\left(T^{-\alpha}_{[nt(1-\epsilon)]}\right)-\underset{j\ge\sqrt{V_n}}{\min}u(j)\mathbb{E}\left(T_{[nt]}^{-\alpha}\right)\right]dt\\
&&\le
\frac{2n^2}{(V_n)^2}\int_\frac{(V_n)^{\frac{1}{2}}}{n}^1
\int_\frac{h_0}{n}^1
\mathbb{E}\left(\sigma_X(T_{[ns]})\right)ds
\left[\underset{j\ge\sqrt{V_n}}{\max}u(j)\mathbb{E}\left(T^{-\alpha}_{[nt(1-\epsilon)]}\right)-\underset{j\ge\sqrt{V_n}}{\min}u(j)\mathbb{E}\left(T_{[nt]}^{-\alpha}\right)\right]dt\\
&&\le\frac{2n}{V_n}\int_\frac{(V_n)^{\frac{1}{2}}}{n}^1\left[\underset{j\ge\sqrt{V_n}}{\max}u(j)\mathbb{E}\left(T^{-\alpha}_{[nt(1-\epsilon)]}\right)-\underset{j\ge\sqrt{V_n}}{\min}u(j)\mathbb{E}\left(T_{[nt]}^{-\alpha}\right)\right]dt\\
&&=
\frac{2n}{V_n}\int_\frac{(V_n)^{\frac{1}{2}}}{n}^1
\left[\underset{j\ge\sqrt{V_n}}{\max}u(j)\frac{1}{1+[nt(1-\epsilon]}g_1([nt(1-\epsilon])-\underset{j\ge\sqrt{V_n}}{\min}u(j)\frac{1}{1+[nt]}g_1([nt])\right]dt\\
&&\le2M\Bigg[\frac{1}{V_n}\left(\int_{\frac{(V_n)^{\frac{1}{2}}}{n}}^1\frac{n}{nt(1-\epsilon)}\right)\left(\underset{j\ge\sqrt{V_n}}{\max}u(j)\underset{i\ge\sqrt{V_n}/2}{\max}g_1(i)\right)\\
&&-\frac{1}{V_n}\left(\int_{\frac{(V_n)^{\frac{1}{2}}}{n}}^1\frac{n}{1+nt}\right)\left(\underset{j\ge\sqrt{V_n}}{\min}u(j)\underset{i\ge\sqrt{V_n}}{\min}g_1(i)\right)\Bigg]\\
&&\le
\frac{1}{1-\epsilon}\left(\underset{j\ge\sqrt{V_n}/2}{\max}u(j)\right)\left(\underset{i\ge\sqrt{V_n}}{\max}g_1(i)\right)\\
&&-\left(\frac{\log(1+n)}{V_n}-\frac{\log(1+\sqrt{V_n})}{V_n}\right)\left(\underset{j\ge\sqrt{V_n}}{\min}u(j)\right)\left(\underset{i\ge\sqrt{V_n}}{\min}g_1(i)\right)\\
&&\underset{n\to\infty}{\to}\left(\frac{1}{1-\epsilon}-1\right)\tilde C_d\,\mathbb{E}\left(L_1^{2d-1}\right),
\end{eqnarray*}
for $\epsilon<1/2$ and hence
$$
\underset{n\to\infty}{\limsup}\,\,I(\epsilon,n)\le2\epsilon\,C\, \mathbb{E}\left(L_1^{-1}\right).
$$
This completes the proof of (\ref{uniform_kk}).\\
{\bf Proof of (\ref{conditional2}) when $\bm{\alpha=1}$.} 
Normalizing by $(\ell(b_n)/(\ell^*(b_n)b_n))^{2d-1}$, we can rewrite $R_n(T)$ defined in (\ref{conditional2}) as
\begin{eqnarray}
R_n(T)&=&\frac{\int_{[0,1]^2}u(\left\vert T_{[nx]}-T_{[ny]}\right\vert+1)\left(\frac{\ell(b_n)}{\ell^*(b_n)}\frac{\left\vert T_{[nx]}-T_{[ny]}\right\vert+1}{b_n}\right)^{2d-1}dxdy} 
{\int_{[0,1]^2}\mathbb{E}\left[u(\left\vert T_{[nx]}-T_{[ny]}\right\vert+1)\left(\frac{\ell(b_n)}{\ell^*(b_n)}\frac{\left\vert T_{[nx]}-T_{[ny]}\right\vert+1}{b_n}\right)^{2d-1}\right]dxdy}.
\end{eqnarray}
Since the integrand in the numerator of $R_n(T)$ converges in probability to $\tilde C_d$ by (\ref{harmonic}), 
using \cite{CREMERS1986305} and (iv) of Lemma \ref{technical} we obtain that the numerator of $R_n(T)$ converges in probability to $\tilde C_d$. The same (iv) of Lemma \ref{technical} guarantees the uniform integrability of this integrand (with respect of $\lambda^2\otimes P$, where $\lambda$ is the Lebesgue probability measure on the unit interval) and hence the convergence of the denominator of $R_n(T)$ to $\tilde C_d$ and hence the convergence of $R_n(T)$ to 1 in probability. \\\\
{\bf Proof of (\ref{conditional3}).} We can write
\begin{eqnarray}
R_n(T)&=&\frac{\int_{[0,1]^2}u(\left\vert T_{[nx]}-T_{[ny]}\right\vert+1)\left(\frac{\left\vert T_{[nx]}-T_{[ny]}\right\vert+1}{b_n}\right)^{2d-1}dxdy}
{\int_{[0,1]^2}\mathbb{E}\left[u(\left\vert T_{[nx]}-T_{[ny]}\right\vert+1)\left(\frac{\left\vert T_{[nx]}-T_{[ny]}\right\vert+1}{b_n}\right)^{2d-1}\right]dxdy}.
\end{eqnarray}
Clearly from (\ref{fclt}), the finite dimensional distributions of
$$
u(\left\vert T_{[nx]}-T_{[ny]}\right\vert+1)\left(\frac{\left\vert T_{[nx]}-T_{[ny]}\right\vert+1}{b_n}\right)^{2d-1}
$$
converge to those of $\tilde C_d\left(\vert L_x-L_y\vert\right)^{2d-1}$ and (since $u$ is bounded) by (ii) of Lemma \ref{technical}
$$
u(\left\vert T_{[nx]}-T_{[ny]}\right\vert+1)\left(\frac{\left\vert T_{[nx]}-T_{[ny]}\right\vert+1}{b_n}\right)^{2d-1}
$$
is $\lambda^2\otimes P$ uniformly integrable and hence using \cite{CREMERS1986305} we conclude that 
\begin{equation}\label{limit_ind_dist}
\int_{[0,1]^2}u(\left\vert T_{[nx]}-T_{[ny]}\right\vert+1)\left(\frac{\left\vert T_{[nx]}-T_{[ny]}\right\vert+1}{b_n}\right)^{2d-1}dxdy\overset{\mathcal{D}}{\to}
\tilde C_d\int_{[0,1]^2}\vert L_x-L_y\vert^{2d-1}dxdy.
\end{equation}
Combining this with (\ref{inverse_levy}) 
and the uniform integrability of the left-hand side in (\ref{limit_ind_dist}) as well as the self-similarity of $L_t$, we obtain that 
\begin{equation*}
\int_{[0,1]^2}\mathbb{E}\left[u(\left\vert T_{[nx]}-T_{[ny]}\right\vert+1)\left(\frac{\left\vert T_{[nx]}-T_{[ny]}\right\vert+1}{b_n}\right)^{2d-1}\right]dxdy\to
\tilde C_d\mathbb{E}\left[\left(\int_{[0,1]^2}\vert L_x-L_y\vert^{2d-1}dxdy\right)\right]
\end{equation*}
and we have 
\begin{eqnarray*}
\tilde C_d\mathbb{E}\left[\left(\int_{[0,1]^2}\vert L_x-L_y\vert^{2d-1}dxdy\right)\right]
&&=2\tilde C_d\left(\int_{0<x<y<1}(y-x)^{\frac{2d-1}{\alpha}}dxdy\right)\mathbb{E}\left(L_1^{2d-1}\right)\\
&&=2\tilde C_d\frac{1}{\left(1+\frac{2d-1}{\alpha}\right)\left(2+\frac{2d-1}{\alpha}\right)}\frac{\Gamma\left(\frac{1-2d}{\alpha}\right)}{\alpha\Gamma(1-2d)}\\
&&=\tilde C_d\frac{1}{(\alpha+2d-1)(2\alpha+2d-1)}\frac{2\alpha\Gamma\left(\frac{1-2d}{\alpha}\right)}{\Gamma(1-2d)}=\frac{\tilde C_d}{C_{\alpha,1-2d}},
\end{eqnarray*}
which shows that $R_n(T)\overset{\mathcal{D}}{\to}Z(\alpha,d)$. It remains to show that $R_n(T)$ and $S'_n(X,T)$ are asymptotically independent when $1-2d<\alpha<1$. We have
\begin{eqnarray*}
P\left(R_n(T)\le u,S_n'(X,T)\le v\right)
=\mathbb{E}\left(1_{\{R_n(T)\le u\}}1_{\{S'_n(X,T)\le v\}}\right)
=\mathbb{E}\left(1_{\{R_n(T)\le u\}}\mathbb{E}\left[1_{\{S'_n(X,T)\le v\}}\vert T\right]\right).
\end{eqnarray*}
Since
$
1_{\{R_n(T)\le u\}}\overset{\mathcal{D}}{\to}1_{\{Z(\alpha,d)\le u\}}
$
and
$
\mathbb{E}\left[1_{\{S'_n(X,T)\le v\}}\vert T\right]\overset{P}{\to}P(N\le v),
$ (from (\ref{almost sure})),
we obtain by Slutsky Theorem that 
$
1_{\{R_n(T)\le u\}}\mathbb{E}\left[1_{\{S'_n(X,T)\le v\}}\vert T\right]\overset{\mathcal{D}}{\to}1_{\{Z(\alpha,d)\le u\}}P(N\le v),
$
which, with Bounded Convergence Theorem, implies that
$$
\mathbb{E}\left(1_{\{R_n(T)\le u\}}\mathbb{E}\left[1_{\{S'_n(X,T)\le v\}}\vert T\right]\right)\to \mathbb{E}\left(1_{\{Z(\alpha,d)\le u\}}P(N\le v)\right)=P(Z(\alpha,d)\le u)P(N\le v).
$$
{\bf Proof of (\ref{continuity}).}
First, we show the continuity (from the right) at $\alpha=1-2d$. It is easy to check that $\mathbb{E}(Z(\alpha,d))=1$. For simplicity let $r=1-2d$. We show that as $\alpha\downarrow r$ then $\textrm{Var}(Z(\alpha,d))\to0$. Noting that $L_t$ is nondecreasing, $\alpha$-self-similar with stationary and independent increments, we can write for $\alpha>r$ 
 \begin{eqnarray*}\lefteqn{
\textrm{Var}(Z(\alpha,d))}\\
&&=4C^2_{\alpha,r}\int_{0<s_1<s_2<1}\int_{0<t_1<t_2<1}\left[\mathbb{E}\left(\left(L_{s_2}-L_{s_1}\right)^{-r}
 \left(L_{t_2}-L_{t_1}\right)^{-r}\right)
 -\mathbb{E}(L_{s_2-s_1}^{-r})\mathbb{E}(L_{t_2-t_1}^{-r})\right]dt_1dt_2ds_1ds_2\\
 &&=
 8C^2_{\alpha,r}\int_{0<s_1<t_1<t_2<s_2<1}\left[\mathbb{E}\left(\left(L_{s_2}-L_{s_1}\right)^{-r}
 \left(L_{t_2}-L_{t_1}\right)^{-r}\right)
 -\mathbb{E}(L_{s_2-s_1}^{-r})\mathbb{E}(L_{t_2-t_1}^{-r})\right]dt_1dt_2ds_1ds_2\\
 &&+
 8C^2_{\alpha,r}\int_{0<t_1<s_1<t_2<s_2<1}\left[\mathbb{E}\left(\left(L_{s_2}-L_{s_1}\right)^{-r}
 \left(L_{t_2}-L_{t_1}\right)^{-r}\right)
 -\mathbb{E}(L_{s_2-s_1}^{-r})\mathbb{E}(L_{t_2-t_1}^{-r})\right]dt_1dt_2ds_1ds_2\\
 &&:=
 I(r,\alpha)+II(r,\alpha).
 \end{eqnarray*}
 Putting $t=t_2-t_1$ and $s=s_2-s_1$ and observing that 
 $$
 \{0<s_1<t_1<t_2<s_2<1\}\subset\{0<s_2-s<t_1<s_2<1\}\cap\{0<t<s<1\},
 $$
 we obtain that
 \begin{eqnarray*}
 \frac{1}{8}I(r,\alpha)
 &\le& 
C^2_{\alpha,r}\int_{0<s_1<t_1<t_2<s_2<1}
\mathbb{E}\left(L_{t_2}-L_{t_1}\right)^{-r}
\Bigg[\mathbb{E}\left(\left(
L_{t_1}-L_{s_1}+L_{s_2}-L_{t_2}
\right)^{-r}\right)-\mathbb{E}(L_{s_2-s_1}^{-r})\Bigg]dtds\\
&=&
C^2_{\alpha,r}\int_{0<s_1<t_1<t_2<s_2<1}
\mathbb{E}\left(L_{t_2}-L_{t_1}\right)^{-r}
\Bigg[\mathbb{E}\left(\left(
L_{t_1}-L_{s_1}+L_{t_1+s_2-t_2}-L_{t_1}
\right)^{-r}\right)-\mathbb{E}(L_{s_2-s_1}^{-r})\Bigg]dtds\\
&=&
C^2_{\alpha,r}\int_{0<s_1<t_1<t_2<s_2<1}
\mathbb{E}\left(L_{t_2}-L_{t_1}\right)^{-r}
\Bigg[\mathbb{E}\left(\left(
L_{s_2-s_1-(t_2-t_1)}
\right)^{-r}\right)-\mathbb{E}(L_{s_2-s_1}^{-r})\Bigg]dtds\\
&\le&
C^2_{\alpha,r}\mathbb{E}(L_1^{-r})\int_{0<t<s<1}
t^{-\frac{r}{\alpha}}\left[(s-t)^{-\frac{r}{\alpha}}-s^{-\frac{r}{\alpha}}\right]sdsdt\\
&=&
C^2_{\alpha,r}\mathbb{E}(L_1^{-r})\int_{0<t<s<1}
t^{-\frac{r}{\alpha}}\left[(s-t)^{1-\frac{r}{\alpha}}-s^{1-\frac{r}{\alpha}}\right]dsdt\\
&&+C^2_{\alpha,r}\mathbb{E}(L_1^{-r})\int_{0<t<s<1}
t^{1-\frac{r}{\alpha}}(s-t)^{-\frac{r}{\alpha}}dsdt\\
&\le&2C^2_{\alpha,r}\mathbb{E}(L_1^{-r})\frac{1}{1-\frac{r}{\alpha}}\to0\qquad\textrm{as }\alpha\downarrow r,
 \end{eqnarray*}
 since when $\alpha\downarrow r$, $C_{\alpha,r}=O\left(1-\frac{r}{\alpha}\right)$ and $\mathbb{E}(L_1^{-r})=\frac{\Gamma\left(\frac{r}{\alpha}\right)}{\alpha\Gamma(r)}\to1/(r\Gamma(r))<\infty$
 by (\ref{down1}) and (\ref{inverse_levy}).\\
Let $u=pt_2+qs_1$ where $p+q=1$, $p,q>0$ fixed. Then with $h_1=s_1-t_1$, $h=t_2-s_1$ and $h_2=s_2-t_2$, we can write
 \begin{eqnarray}\lefteqn{
\frac{1}{8}II(r,\alpha)}\nonumber\\
&&\le C^2_{\alpha,r}\int_{0<t_1<s_1<t_2<s_2<1}
\Bigg[\mathbb{E}\left(\left(L_{u}-L_{t_1}\right)^{-r}
 \left(L_{s_2}-L_u\right)^{-r}\right)
 -\mathbb{E}(L_{s_2-s_1}^{-r})\mathbb{E}(L_{t_2-t_1}^{-r})\Bigg]dt_1dt_2ds_1ds_2\nonumber\\
 &&=C^2_{\alpha,r}\mathbb{E}(L_1^{-r})
 \int_{0<h+h_1+h_2<1}
 \left((h_1+ph)^{-\frac{r}{\alpha}}(h_2+qh)^{-\frac{r}{\alpha}}
 -(h_1+h)^{-\frac{r}{\alpha}}(h_2+h)^{-\frac{r}{\alpha}}
 \right)dhdh_1dh_2\nonumber\\
 &&=
C^2_{\alpha,r}\mathbb{E}(L_1^{-r})
 \int_{0<h+h_1+h_2<1}
 (h_1+ph)^{-\frac{r}{\alpha}}\left[(h_2+qh)^{-\frac{r}{\alpha}}
 -(h_2+h)^{-\frac{r}{\alpha}}\right]
 dh_1dh_2dh\label{int11}\\
 &&+C^2_{\alpha,r}\mathbb{E}(L_1^{-r})
 \int_{0<h+h_1+h_2<1}
 (h_2+h)^{-\frac{r}{\alpha}}\left[(h_1+ph)^{-\frac{r}{\alpha}}
 -(h_1+h)^{-\frac{r}{\alpha}}\right]
 dh_2dh_1dh\label{int22}
 \end{eqnarray}
 Integrating with respect to $h_1$ then using the mean value Theorem, the integral in (\ref{int11}) can be bounded by
 \begin{eqnarray*}
\int_0^1\int_0^{1-h}h^{1-\frac{r}{\alpha}}(h_2+qh)^{-1-\frac{r}{\alpha}}dh_2dh
&=&\frac{\alpha}{r}\int_0^1h^{1-\frac{r}{\alpha}}\left((qh)^{-\frac{r}{\alpha}}-(1-ph)^{-\frac{r}{\alpha}}\right)dh\\
&\le&
\frac{\alpha}{r}\int_0^1h^{1-\frac{r}{\alpha}}(qh)^{-\frac{r}{\alpha}}dh=
q^\frac{\alpha}{r}\frac{r}{2\alpha}\frac{1}{1-\frac{r}{\alpha}}.
 \end{eqnarray*}
 The integral in (\ref{int22}) treats the same way. Therefore we conclude that $II(r,\alpha)\to0$ as $\alpha\downarrow r$ and hence $\textrm{Var}(Z(\alpha,d))\to$ as $\alpha\downarrow r.$\\
The proof that $\textrm{Var}(Z(\alpha,d))\to0$ as $\alpha\uparrow1$ is straightforward. Actually,
 We easily see from (\ref{inverse_levy}) that as $\alpha\uparrow1$, 
 $\mathbb{E}(L_1^{-kr})\to1$ for all $k\ge1$, so that 
 $\textrm{Var}(L_1^{-r})\to0$.
 Hence, using Cauchy-Schartz inequelity and self-similarity, we obtain that, as $\alpha\uparrow1,$
 \begin{eqnarray*}
\textrm{Var}(Z(\alpha,d))\le \left(\frac{C_{\alpha,r}}{1-\frac{r}{\alpha}}\right)^2\textrm{Var}(L_1^{-r})\to0.
 \end{eqnarray*}
{\bf Proof of (\ref{moments_problem})}
 Using the generalized Minckowski inequality we have (with $r=1-2d$)
$$
\mathbb{E}(L_{t_1}^{-r}\cdots L_{t_k}^{-r})\le\prod_{i=1}^k\mathbb{E}\left(L^{-rk}_{t_i}\right)^{\frac{1}{k}}=\left(\prod_{i=1}^kt_i^{-\frac{r}{\alpha}}\right)\mathbb{E}\left(L_1^{-rk}\right)
$$
and hence by Fubini Theorem we get
\begin{eqnarray*}\lefteqn{
\nu_k:=\mathbb{E}\left[\left(\int_{[0,1]^2}\vert L_t-L_s\vert^{-r}dtds\right)^k\right]}\\
&&=\int_{[0,1]^{2k}}\mathbb{E}(\vert L_{t_1}-L_{s_1}\vert^{-r}\cdots \vert L_{t_k}-L_{s_k}\vert^{-r})dt_1\cdots dt_kds_1\cdots ds_k\\
&&\le\left(\frac{1}{d(2d+1)}\right)^k\mathbb{E}\left(L_1^{-rk}\right),
\end{eqnarray*}
using (\ref{value in d}).
Hence
\begin{eqnarray*}
\sum_{k=1}^\infty\nu_k^{{-1/k}}\ge d(2d+1)\left(\mathbb{E}\left(L_1^{-rk}\right)\right)^{-1/k}.
\end{eqnarray*}
From (\ref{inverse_levy}) we get
$$
\left(\mathbb{E}\left(L_1^{-rk}\right)\right)^{-1/k}=\left(\frac{\alpha\Gamma(rk)}{\Gamma(rk/\alpha)}\right)^{1/k}.
$$
Using \cite{Anderson1997AMP} bounds of the $\Gamma$ function , we have for all $x>0$
$$
x^{(1-\gamma)x-1}<\Gamma(x)<x^{x-1},
$$
where $\gamma$ is the Euler-Mascheroni constant ($\gamma=.577...)$, and hence we get
$$
\left(\frac{\alpha\Gamma(rk)}{\Gamma(rk/\alpha)}\right)^{1/k}>
\left(\frac{\alpha(rk)^{(1-\gamma)rk-1}}{\left(\frac{rk}{\alpha}\right)^{rk/\alpha-1}}\right)^{1/k}
=\frac{(rk)^{(1-\gamma)r}}{\left(\frac{rk}{\alpha}\right)^{r/\alpha}}=Ck^{\left((1-\gamma)-\frac{1}{\alpha}\right)r}\ge Ck^{-\frac{r}{\alpha}}
$$
which is not summable (since $r/\alpha<1$)
and hence
$$
\sum_{k=1}^\infty\nu_k^{{-1/k}}=\infty,
$$
which implies that $(\nu_k)_{k\ge1}$ determine the distribution of 
$$
\int_{[0,1]^2}\vert L_t-L_s\vert^{-r}dtds.
$$
 \section{Proof of the Lemmas \ref{technical} and \ref{lem1}}

 {\bf Proof of Lemma \ref{technical}}
(i) Using the fact that $M_n+1>1$ and that $\ln(1-x)\le-x$ for $0<x<1$, we can write
\begin{eqnarray*}
\mathbb{E}\left(\frac{M_n+1}{b_n}\right)^{-r}&=&\int_0^{b_n^r} P(M_n+1<b_nx^{-1/r})dx\le1+\int_1^{b_n^r} P(M_n<b_nx^{-1/r})dx\\
&=&1+\int_1^{b_n^r} [P(T_1\le b_nx^{-1/r})]^ndx=1+\int_1^{b_n^r} \left[1-\left(b_nx^{-1/r}\right)^{-\alpha}\ell\left(b_nx^{-1/r}\right)\right]^ndx\\
&\le& 1+\int_1^{b_n^r} \exp\left(-nb_n^{-\alpha}x^{\alpha/r}\ell\left(b_nx^{-1/r}\right)\right)dx=1+\int_1^{b_n^r}\exp\left(-x^{\alpha/r}\frac{\ell\left(b_nx^{-1/r}\right)}{\ell\left(b_n\right)}\right)dx.
\end{eqnarray*}
Also, it is clear from (\ref{karamata1}) that $\ell(x)$ 
is bounded away from zero and infinity on every compact interval of $[1,\infty)$ and therefore by Potter's bound (see for instance \cite{bingham1987large} Theorem 1.5.6. ii), we obtain that for every $\delta>0$ there exists a positive constant $C_\delta$ such that for every $n$ and $x\in[1,{b_n^r}]$,
$$ 
\frac{\ell\left(b_nx^{-1/r}\right)}{\ell\left(b_n\right)}\ge C_\delta x^{-\frac{\delta}{r}}
$$ and hence, taking $\delta<\alpha$, the last integral above is bounded by
$$ 
\int_1^\infty e^{-C_\delta x^\frac{\alpha-\delta}{r}}dx<\infty.
$$
(ii)
Using the fact that $b_n$ is the inverse function of a nondecreasing regularly function with index $\alpha$, we obtain that $b_n$ is nondecreasing regularly varying with index $1/\alpha$ and is a slowly varying function (see for instance \cite{resnick2007heavy} Proposition 2.6 (v)) that is bounded away from zero and infinity on compact intervals. Also, using the fact that for $x>y$, $[nx]-[ny]+1\ge n(x-y)$ then with similar argument as above, we obtain that for for every $\delta>0$ there exists $C_\delta>0$ such that for all $n$ and $x,y\in[0,1]$
$$
\left(\frac{b_{[nx]-[ny]+1}}{b_n}\right)^{-r}\le C_\delta \vert x-y\vert^{-r/\alpha-\delta}
$$
which is integrable for small $\delta$. 
Moreover, it is easy to check that $b_{k+1}/b_k\le C_1$ for some positive constant $C_1$ for all $k\ge1$. Denoting 
$D_n=\{x,y\in[0,1],\,\,x>y+1/n\}$, we can write for some $r<\alpha$
$$
\int_{[0,1]^2}\mathbb{E}\left(\frac{\left\vert T_{[nx]}-T_{[ny]}\right\vert+1}{b_n}\right)^{-r}dxdy
\le 2\frac{b_n^r}{n}+2\int_{D_n}\mathbb{E}\left(\frac{\left\vert T_{[nx]}-T_{[ny]}\right\vert+1}{b_n}\right)^{-r}dxdy.
$$
Note that $b_n^r/n\to0$ and hence bounded, and that on $D_n$, $[nx]-[ny]\ge1$, so that, using (i) to bound the expected value in the right-hand side below, we get
\begin{eqnarray*}
\int_{D_n}\mathbb{E}\left(\frac{T_{[nx]}-T_{[ny]}+1}{b_n}\right)^{-r}dxdy
&\le&\int_{D_n}\mathbb{E}\left(\frac{T_{[nx]-[ny]}+1}{b_{[nx]-[ny]}}\right)^{-r}
\left(\frac{b_{[nx]-[ny]}}{b_{[nx]-[ny]+1}}\frac{b_{[nx]-[ny]+1}}{b_n}\right)^{-r}dxdy\\
&\le&CC_1C_\delta\int_{[0,1]^2}\vert x-y\vert^{-r/\alpha-\delta}=\frac{CC_\delta}{1-r/\alpha-\delta}\frac{1}{2-r/\alpha-\delta},
\end{eqnarray*}
which completes the proof of (ii) of the Lemma \ref{technical}.\\
(iii)
We have
$$
\frac{T_n+1}{b_n}=\sum_{j=1}^{n/2}\frac{\Delta_j}{b_n}+\sum_{j=n/2+1}^n\frac{\Delta_j}{b_n}+\frac{1}{b_n}\ge
\sum_{j=1}^{n/2}\frac{\Delta_j}{b_n}1_{\{\Delta_j\le b_n\}}+
\frac{M^*_n+1}{b_n}
$$
where 
$$
M^*_n=\underset{n/2+1\le j\le n}{\max}(\Delta_j).
$$
Let us evaluate the mean and the variance of the truncated sum above. Using (\ref{heavy tail}), (\ref{elstar}) and (\ref{bn}) we can write
$$
\mu_n:=\mathbb{E}\left(\sum_{j=1}^{n/2}\frac{\Delta_j}{b_n}1_{\{\Delta_j\le b_n\}}\right)
=\frac{n}{2b_n}\left(\ell^*(b_n)-\ell(b_n)\right)=\frac{\ell^*(b_n)}{2\ell(b_n)}-\frac{1}{2}\to\infty
$$
and we can easily check that
$$ 
\frac{n}{2}\mathbb{E}\left[\left(\frac{\Delta_1}{b_n}\right)^21_{\{\Delta_1\le b_n\}}\right]=\frac{1}{2},
$$
so that (using the fact that the variables are i.i.d.), we have
$$
\sigma^2_n:=\textrm{Var}\left(\sum_{j=1}^{n/2}\frac{\Delta_j}{b_n}1_{\{\Delta_j\le b_n\}}\right)
=\frac{n}{2}\textrm{Var}\left(\frac{\Delta_1}{b_n}1_{\{\Delta_1\le b_n\}}\right)= \frac{n}{2}\mathbb{E}\left[\left(\frac{\Delta_1}{b_n}\right)^21_{\{\Delta_1\le b_n\}}\right]-\frac{4\mu_n^2}{n}\to\frac{1}{2}.
$$
Now using \cite{PITTENGER199091} upper bounds for inverse moments, we obtain (using the fact that $M^*_n$ is independent from $\Delta_j$, $j\le n/2$) that 
\begin{eqnarray*}\mathbb{E}\left[\left(\frac{T_n+1}{b_n}\right)^{-2}\right]&\le&\mathbb{E}\left[\left(\sum_{j=1}^{n/2}\frac{\Delta_j}{b_n}1_{\{\Delta_j\le b_n\}}+\frac{M^*_n+1}{b_n}\right)^{-2}\right]\\
&=&
\mathbb{E}\left(\mathbb{E}\left[\left(\sum_{j=1}^{n/2}\frac{\Delta_j}{b_n}1_{\{\Delta_j\le b_n\}}+\frac{M^*_n+1}{b_n}\right)^{-2}\bigg\vert M^*_n\right]\right)\\
&\le&\mathbb{E}\left(\frac{\sigma_n^2}{\sigma^2_n+\mu_n^2}\left(\frac{M^*_n+1}{b_n}\right)^{-2}\right)+\mu_n^{-2}\\
&=&\left(\frac{\sigma_n^2}{\sigma^2_n+\mu_n^2}\right)\mathbb{E}\left(\left(\frac{M^*_n+1}{b_n}\right)^{-2}\right)+\mu_n^{-2}
\end{eqnarray*}
and hence
$$
\underset{n}{\sup}\,
\mathbb{E}\left[\left(\frac{T_n+1}{b_n\mu_n}\right)^{-2}\right]\le \underset{n}{\sup}\,\left(\frac{\sigma_n^2\mu_n^2}{\sigma^2_n+\mu_n^2}\right)\,\underset{n}{\sup}\,\mathbb{E}\left(\left(\frac{M^*_n+1}{b_n}\right)^{-2}\right)+1<\infty 
$$
as $\sigma^2_n$ is bounded and $\mu_n\to\infty$, 
which completes the proof of (iii). \\
(iv) Similar to the (ii), we have
\begin{eqnarray}\label{vague1}\lefteqn{
\int_{[0,1]^2}\mathbb{E}\left[\left(\frac{T_{\vert[nx]-[ny]\vert}+1}{b_n}\frac{\ell(b_n)}{\ell^*(b_n)}\right)^{-r}\right]dxdy}\nonumber\\
&&\le 2\frac{b_n^r}{n}\left(\frac{\ell(b_n)}{\ell^*(b_n)}\right)^{-r}
+2\int_{D_n}\mathbb{E}\left(\frac{T_{\vert[nx]-[ny]\vert}+1}{b_n}\frac{\ell(b_n)}{\ell^*(b_n)}\right)^{-r}dxdy.
\end{eqnarray} 
Since $b_n=nh(n)$ with $h$ a slowly varying function at infinity, we have for $0<r<1$,
$$
\frac{b_n^r}{n}\left(\frac{\ell(b_n)}{\ell^*(b_n)}\right)^{-r}\to0
$$
and hence bounded.
From (iii), the integrand in (\ref{vague1}) is uniformly bounded over $D_n$ and hence the second term in (\ref{vague1}) is also uniformly bounded. This concludes the proof of (iv).\\\\
 {\bf Proof of Lemma \ref{lem1}}\\
{\bf Part (i)}
 According to \eqref{eq:a} we have $\sigma_X(h) \sim \tilde C_d h^{2d-1}$ with $\tilde C_d >0 $ so the autocovariance function becomes positive starting from some lag $m-1$. Also $\sigma_Y(h) = \mathbb{E} (\sigma_X(T_h) )>0$ for $h\geq m-1 $ since $T_h\geq h$. 
We have $$0 \le \textrm{Var}\left ( \frac{1}{\sqrt{m}} \sum_{j=1}^m Y_j \right) = \sum_{j=-m}^m \sigma_Y(j) \left(1- \frac jm \right). $$
\begin{enumerate}
 \item Case 1 : $\sum_{h=-\infty}^\infty |\sigma_Y(h) | < \infty $. \\ 
We have 
$$0\le \sum_{j=-m}^m \sigma_Y(j) \left(1- \frac jm \right) \to 
\sum_{h=-\infty}^\infty \sigma_Y(h).$$ 
If this limit is positive, it is immediate that there exists $m$ such that $ \sum_{j= -m } ^m \sigma_Y(j) >0. $
If $\sum_{h=-\infty}^\infty \sigma_Y(h) = 0$ , since $\sigma_Y(h) >0$ for $h>m$ then $\sum_{j=-m }^m \sigma_Y(j) \left(1- \frac jm \right) < 0$ which is a contradiction. 
\item Case 2 : $\sum_{h=-\infty}^\infty |\sigma_Y(h) | = \infty $ \\ 
Since $\sigma_Y(h) >0$ for $h>m$, we must have $\sum_{h=-\infty}^\infty \sigma_Y(h) = \infty $ which implies that there exists $m_1$ such that 
 $\sum_{h=-m_1}^{m_1} \sigma_Y(h) >0 $. 
\end{enumerate}
This completes the proof of part (i) of Lemma \ref{lem1}.\\
{\bf Part (ii)}
{\bf Step 1 } 
First of all, for $j<T_n$, positive integer, let $k(j)=\inf\{k\ge1:\quad T_k\ge j\}$. 
Since $a_k\sim ck^{d-1}$ as $k\to\infty$, then $a_k/((k+1) ^{d-1})$ is bounded and hence for $j\geq k(j)$, $$ |a_{T_k-j}| \leq C (T_k-j+1)^{d-1} \leq C (T_k- T_{k(j)}) +1)^{d-1} 
 \leq C (k- {k(j}) +1)^{d-1} 
$$ 
$$ 
\vert d_{n,j}\vert = \left\vert\sum_{k={k(j)}}^na_{T_k-j}\right\vert \leq \sum_{k=k(j)}^n C (k- {k(j}) +1)^{d-1} 
\leq \sum_{k=1}^n C k^{d-1} \sim C n^{d} .
$$ 

{\bf Step 2}

\begin{eqnarray*}\lefteqn{
\textrm{Var}\left(\sum_{k=1}^nY_k\vert T_1,\ldots,T_n\right)=\sigma^2\left(\sum_{j\in\mathbb{Z}}d^2_{n,j}\right)}\\
&&=\sum_{k=1}^n\sum_{k'=1}^n\sigma_X(T_k-T_{k'})=
n^2\int_{[0,1]^2}\sigma_X(T_{[nx]}-T_{[ny]})dxdy.
\end{eqnarray*}
For all $(x,y)\in [0 , 1]^2$ such that $|x-y| > m/n $ we have 
$$ |T_{[nx]}-T_{[ny]}| \geq |[nx]-[ny]| \geq m - 1, $$
and hence by part (i) of Lemma \ref{lem1}, $\sigma_X(T_{[nx]}-T_{[ny]}) >0 $. 
We have 

 \begin{eqnarray*}
n^2 \int_{[0,1]^2} \sigma_X(T_{[nx]}-T_{[ny]})dxdy &=& n^2 \int_{ |x-y| \leq \frac{m}{n}} \sigma_X(T_{[nx]}-T_{[ny]})dxdy +n^2 \int_{ |x-y| > \frac{m}{n}} \sigma_X(T_{[nx]}-T_{[ny]})dxdy \\ 
&\geq& n^2 \int_{ |x-y| \leq \frac{m}{n}} \sigma_X(T_{[nx]}-T_{[ny]})dxdy \\
&=& n^2 \sum_{j= - 
m}^m \sum_{k=
1}^{n - |j|} \sigma_X(\Delta_k +\cdots + \Delta_{k+j}) \, \frac{1}{n^2
} \\ 
&=& {n}\sum_{j= - 
m}^m \frac{1}{n} \sum_{k=
1}^{n - |j|} \sigma_X(\Delta_k +\cdots + \Delta_{k+j}) .
\end{eqnarray*}
For fixed $j\le m$, $(\sigma_X(\Delta_k +\cdots + \Delta_{k+j})_k$ is $j$-dependent stationary sequence and hence satisfies the Law of Large Numbers: 
$$\frac{1}{n} \sum_{k=
1}^{n - |j|} \sigma_X(\Delta_k +\cdots + \Delta_{k+j}) \overset{a.s.}{\to} \mathbb{E}( \sigma_X( \Delta_1+\cdots+\Delta_j)) = \sigma_Y(j).$$
Hence 
$$ \sum_{j= - 
m}^m \frac{1}{n} \sum_{k=
1}^{n - |j|} \sigma_X(\Delta_k +\cdots + \Delta_{k+j}) \overset{a.s.}{\to} \sum_{j= -m }^m \sigma_Y(j) >0. 
$$ 
 
In conclusion we obtain that 
$$
\frac{1}{n ^{2d }} \textrm{Var}\left(\sum_{k=1}^nY_k\vert T_1,\ldots,T_n\right) \overset{a.s.}{\to} \infty 
$$
which concludes the proof of part (ii) of the Lemma.\\
\bibliographystyle{apalike}
\bibliography{subsampling}

\end{document}